\documentclass[reqno,11pt,a4paper]{amsart}
\usepackage{amssymb,amsmath,amsthm}
\usepackage{extarrows}
\usepackage{color}

\newcommand{\R}{\mathbb{R}}

\newtheorem{theorem}{Theorem}[section]
\newtheorem{lemma}[theorem]{Lemma}
\newtheorem{proposition}[theorem]{Proposition}
\newtheorem{corollary}[theorem]{Corollary}

\theoremstyle{definition}

\numberwithin{equation}{section}

\begin{document}

  \title{ Torsion energy with boundary mean zero condition  }

\author[Q. Li]{Qinfeng Li}
\address{Qinfeng Li, School of Mathematics, Hunan University, Changsha, Hunan, China}
\email{liqinfeng1989@gmail.com}

\author[W. Xie]{Weihong Xie}
\address{Weihong Xie, School of Mathematics and Statistics, HNP-LAMA, Central South University, Changsha, Hunan,   China}
\email{wh.xie@csu.edu.cn}

\author[H. Yang]{Hang Yang}
\address{Hang Yang, School of Mathematics, Hunan University, Changsha, Hunan, China}
\email{yanghang0925@163.com}

\begin{abstract}
Motivated by establishing Neumann Talenti type comparison results, we concern the minimization of the following shape functional under volume constraint:
  \begin{align*}
        T(\Omega):=\inf\left\{\frac12 \int_{\Omega} |\nabla u|^2\,dx -\int_\Omega  u\,dx: u\in H^1(\Omega),\ \int_{\partial \Omega}ud\sigma=0 \right\}.
  \end{align*}
We prove that ball is a local minimizer to $T(\cdot)$ under smooth perturbation, but quite surprisingly, ball is not locally minimal to $T(\cdot)$ under Lipschitz perturbation. In fact, let $P_N$ be the regular polygon in $\mathbb{R}^2$ with $N$ sides and area $\pi$, then we prove that $T(P_N)$ is a strictly increasing function with respect to $N$ and $\lim_{N\rightarrow \infty}T(P_N)=T(B)$ where $B$ is the unit disk.

As another side result, we prove that in dimension bigger than or equal to three, rigidity results of Serrin's seminal overdetermined system is not stable under Dirichlet perturbations, in contrast to the stability of rigidity under Neumann perturbation.
\end{abstract}
\date{}

\maketitle

\vskip3mm


 \vskip6mm

\section{Introduction and statement of  results}

\subsection{Background and motivating problems}
In this paper, we consider the following shape functional defined on bounded Lipschitz domains in $\R^n$:
\begin{equation}\label{T-beta}
T_\beta(\Omega):=\inf\left\{\frac12 \int_{\Omega} |\nabla u|^2\, dx+\frac\beta2\int_{\partial \Omega}u^2\, d\sigma-\int_\Omega fu\, dx: u\in H^1(\Omega),\ \int_{\partial \Omega}ud\sigma=0 \right\},
\end{equation}
where $\beta\ge 0$ is a constant parameter and $f>0$ is a radial weight function. 

It is readily checked that the infimum in \eqref{T-beta} is attained by $u_\Omega$ which satisfies
\begin{equation}\label{EL-Tb}
    \begin{cases}
     -\Delta u_\Omega=f \qquad &\text{in }\Omega;\\
     \frac{\partial u_\Omega}{\partial \nu}+\beta u_\Omega =c \qquad &\text{on }\partial\Omega;\\
     \int_{\partial \Omega}u_\Omega d\sigma=0.
  \end{cases}
\end{equation}
Here $\nu$ is the outer unit normal to $\partial\Omega$ and $c$ is a constant satisfying compatibility of \eqref{EL-Tb}. Hence $c=-\frac{\int_\Omega f\, dx}{P(\Omega)}$. When $\beta>0$, $\eqref{EL-Tb}_{2}$ automatically implies $\eqref{EL-Tb}_3$, while when $\beta=0$, they are independent. Both cases the system \eqref{EL-Tb} admits a unique solution.

As $\beta\rightarrow \infty$, the    functional $T_\beta$ is reduced to
\begin{equation}\label{T-infty}
  T_\infty(\Omega):=\inf\left\{\frac12 \int_{\Omega} |\nabla u|^2\, dx -\int_\Omega fu\, dx: u\in H^1_0(\Omega) \right\}.
\end{equation}
The corresponding minimizer $u$ solves
\begin{equation}\label{EL-D}
    \begin{cases}
     -\Delta u=f \qquad &\text{in }\Omega;\\
     u=0 \qquad &\text{on }\partial\Omega.
  \end{cases}
\end{equation}
The seminal Saint-Venant inequality states that when $f\equiv 1$, among all shapes with fixed volume, $T_\infty(\cdot)$ attains its minimum at round shape. Later, it is shown in \cite{Li-Yang} that the same results hold if $f$ is radially decreasing, which is essentially a consequence of a stronger result proved by Talenti in his celebrated paper \cite{Talenti}, where he established the pointwise comparison result:
 \begin{equation}\label{Talenti-comparison}
   u^\sharp(x)\le v(x) \quad \mbox{ for all } x\in\Omega^\sharp.
 \end{equation}
In the above, $\Omega^\sharp$  represents the ball of the same volume as $\Omega$ centered
at the origin, $u^\sharp$ is the decreasing Schwarz  rearrangement of $u$, and $v$ satisfies
\begin{equation*}
      \begin{cases}
     -\Delta v=f^\sharp \qquad &\text{in }\Omega^\sharp;\\
     v=0 \qquad &\text{on }\partial\Omega^\sharp.
  \end{cases}
\end{equation*}
The equality case is obtained in \cite{Alvino-Lions-Trombetti}.

Remarkably, Alvino,   Nitsch  and   Trombetti in \cite{Alvino-Nitsch-Trombetti} established Talenti type comparison results under Robin boundary condition. They study the following   Robin boundary  problems:
 \begin{equation}\label{EL-R}
    \begin{cases}
     -\Delta u=f \qquad &\text{in }\Omega;\\
      \frac{\partial u}{\partial \nu}+\beta u=0 \qquad &\text{on }\partial\Omega.
  \end{cases}
\end{equation}
 and
 \begin{equation*}
      \begin{cases}
     -\Delta v=f^\sharp \qquad &\text{in }\Omega^\sharp;\\
    \frac{\partial v}{\partial \nu}+\beta v=0 \qquad &\text{on }\partial\Omega^\sharp;
  \end{cases}
\end{equation*}
 and showed that  for any nonnegative function $f\in L^2(\Omega)$ and $n=2$,  the following weaker comparison result holds
\begin{equation}\label{L1-comparison}
  \|u\|_{L^1(\Omega)}\le \|v\|_{L^1(\Omega^\sharp)}.
\end{equation}
On the other hand, when $f\equiv1$, they obtained \eqref{Talenti-comparison} for $n=2$ and \eqref{L1-comparison} for $n\ge 3$, while leave the validity of \eqref{Talenti-comparison} for the case $n\ge 3$ open. Nevertheless, these results are very strong enough and can also give an alternative proof of Bossel-Danners inequality \cite{Bossel,Daners}, different from the free-discontinuity approach in \cite{BG15}. Whereafter,  Alvino et al.    in \cite{Alvino-Chiacchio-Nitsch-Trombetti}  extended the results of \cite{Alvino-Nitsch-Trombetti} by considering the more general boundary condition:
\begin{equation}\label{general bdy}
    \frac{\partial u}{\partial \nu}+\beta(x) u=0 \qquad  \text{on }\partial\Omega
\end{equation}
with $\beta(x)$ not being a constant, and assuming that
\begin{equation*}
  \int_E f \, dx\le\frac{|E|^{\frac{n-2}{n}}}{|\Omega|^{\frac{n-2}{n}}} \int_\Omega f \,dx,\quad\mbox{ for all measurable }E\subset \Omega.
\end{equation*}

 Naturally,  we would like to investigate whether or not there exists a possible version of Talenti type or Saint-Venant type inequality under constant Neumann boundary conditions. Note that all the previous $L^1$ comparison results equivalently says that fixing volume, ball shape is a maximizer to the functional  
 \begin{align}
     \label{Gfunctional}
G(\Omega):=\int_{\Omega}|\nabla u|^2 \, dx,
 \end{align}
 where $u$ either solves \eqref{EL-D} or \eqref{EL-R}. Now we seek Neumann Saint-Venant or Talenti inequality on the shape functional \eqref{Gfunctional}. The reason is that,  solutions to Possion equation with constant Neumann data are not unqiue, so it is not possible to state some $L^1$ comparison results. Nevertheless, the solutions are unique up to a constant, and thus  \eqref{Gfunctional} is the same no matter what different representatives of solutions we choose.

 In order for the convenience of computation, we choose the boundary mean zero condition as our consideration. That is, we consider the maximization of \eqref{Gfunctional}, inside which $u$ satisfies
 \begin{equation}\label{EL-N}
    \begin{cases}
     -\Delta u =1 \qquad &\text{in }\Omega,\\
     \frac{\partial u }{\partial \nu}  =c \qquad &\text{on }\partial\Omega,\\
     \int_{\partial \Omega}u  d\sigma=0,
  \end{cases}
\end{equation}
where $c$ is a constant satisfying compatibility condition, that is, $c=-|\Omega|/P(\Omega)$. We call solution to \eqref{EL-N} the \textit{Neumann torsion function with boundary vanishing mean}.

Now \eqref{EL-N} is exactly the system \eqref{EL-Tb} with $\beta=0$ and $f=1$. Hence maximizing \eqref{Gfunctional}, where $u$ satisfies \eqref{EL-N}, is equivalent to minimizing \eqref{T-beta} for $\beta=0$ and $f\equiv 1$, due to integration by parts. That is, our goal is to minimize \begin{equation}\label{T-function}
     T(\Omega):=\inf\left\{\frac12 \int_{\Omega} |\nabla u|^2dx -\int_\Omega  udx: u\in H^1(\Omega),\ \int_{\partial \Omega}ud\sigma=0 \right\},
  \end{equation}prescribing volume. 

Clearly for any $\Omega$, $T(\Omega)\le 0$. Preliminary computations show that 
\begin{equation*}
      T(\text{cube})<    T(\text{ball})<    T(\text{thin  rectangle})\rightarrow 0.
\end{equation*}
Hence, there is no maximizer to $T(\cdot)$, and a ball is not a minimizer to $ T(\cdot)$.

We remark that by standard scaling argument, minimizing $T(\cdot)$ is equivalent to minimizing 
\begin{align*}
    \kappa_p(\Omega):=\inf\left\{\frac{\int_{\Omega} |\nabla u|^2\, dx}{\left(\int_\Omega  u^p\, dx\right)^{2/p}}: u\in H^1(\Omega)\setminus \{0\},\ \int_{\partial \Omega}ud\sigma=0 \right\}
\end{align*}
for $p=1$. Hence even in the class of convex domains, no maximizer exists for $\kappa_1(\cdot)$ prescribing volume. Nevertheless, it is shown in \cite{Huang-Li-Li-Yao} that ball is the unique maximizer of $\kappa_2(\cdot)$. This makes the study of \eqref{T-function} and its more general form \eqref{T-beta} more interesting.  

\vskip 0.3cm

\subsection{Our results}

\subsubsection{Stability under smooth perturbation}

Even though ball is not a global minimizer to $T(\cdot)$ under volume constraints, we are still interested in whether it is a stationary shape or stable shape under smooth perturbation, from variational point of view. More precisely, we let $F(t,x)$ be the \textit{flow map} generated by a smooth vector field $\eta \in C^{\infty}_0(\mathbb{R}^n, \mathbb{R}^n)$, that is, 
	\begin{align*}
		\begin{cases}
			\frac{d}{dt}F(t,x)=\eta\circ F(t,x)\quad  &t \ne 0\\
			F(0,x)=x\quad & t=0.
		\end{cases}
	\end{align*}
	We also denote $F_t(x):=F(t,x)$, and hence $F_t$ is a local diffeomorphism when $|t|$ is small. We say that $F_t$ or $\eta$ preserves the volume of $\Omega$, if $|F_t(\Omega)|=|\Omega|$. When we say that $\Omega$ is a \textit{stationary} shape to $T(\cdot)$ under smooth perturbation, we mean $\frac{d}{dt}\Big|_{t=0} T(F_t(\Omega))=0$ for any smooth flow map $F_{t}(\cdot)$ preserving the volume of $\Omega$ along the flow. We say $\Omega$ is \textit{stable} to $T(\cdot)$ under smooth perturbation, if $\Omega$ is stationary to $T(\cdot)$ and that $\Omega$ satisfies $\frac{d^2}{dt^2}\Big|_{t=0} T(F_t(\Omega))\ge 0$ for any smooth volume preserving flow map. 

We first state a more general result.

\begin{theorem}\label{the:local}
Let $f > 0$ be a smooth radial function about the origin, and   $B_R\subset \R^n$ be a ball of radius $R$ centered at the origin. Then for any $\beta\ge0$,
the ball $B_R$ is a stationary shape
 of \eqref{T-beta} under smooth perturbation. Moreover,  the ball $B_R$  is also a stable shape of \eqref{T-beta} under smooth perturbation if and only if
 \begin{equation}\label{f-condition}
     \frac{n-1-\beta R}{n}\bar f_{B_R}\le f(R)\le \bar f_{B_R},
 \end{equation}
 where $\bar f_{B_R}=\frac1{|B_R|}\int_{B_R} fdx$ and $f(R)=f\big|_{\partial B_R}$.
\end{theorem}

Clearly, \eqref{f-condition} contains the case of $\beta\ge 0$ and $f \equiv 1$, and thus the following corollary is immediate.
\begin{corollary}
\label{localminsmooth}
    Under smooth volume-preserving perturbations, ball is a local minimizer to $T(\cdot)$ given by \eqref{T-function}.
\end{corollary}

The proof of Theorem \ref{the:local} relies on evolution equations of some geometric quantities along flow, and also takes advantage of Steklov eigenvalue problem, which has been used to study several shape optimizations recently, such as in \cite{Huang-Li-Li} and \cite{Li-Yang}. In \cite{Li-Yang}, the minimization of the following functional is considered:
\begin{equation}\label{J-beta}
J_\beta(\Omega):=\inf\left\{\frac12 \int_{\Omega} |\nabla u|^2dx+\frac\beta2\int_{\partial \Omega}u^2d\sigma-\int_\Omega fudx: u\in H^1(\Omega) \right\}  \quad \mbox{ for } \beta>0.
\end{equation}
In fact, there is a closed connection between $T_\beta(\Omega)$ and $J_\beta(\Omega)$, from which an alternative proof of Corollary \ref{localminsmooth} is obtained. This will be discussed in section 4.

We mention that from the proof of Theorem \ref{the:local}, ball shape is strict local minimal under smooth volume-preserving perturbation, if the perturbation is not a translation.

\vskip 0.3cm

\subsubsection{Instability under Lipschitz perturbation: monotonicity on regular polygon }
As mentioned, we can show that under smooth volume preserving deformations, ball is a local minimizer to $T(\cdot)$. However, rather strikingly, we can also prove that ball is \textbf{not} a local minimizer to $T(\cdot)$ under Lipschitz volume-preserving perturbation, let alone global minimality. Such phenomenon on shape functionals might not have been realized before. 

Our choice of perturbation path is the regular polygons with the same area $\pi$. We have the following result:

\begin{theorem}
    \label{PN}   
Let $P_N$ be the polygon in $\mathbb{R}^2$ with $N$ sides and area $\pi$, and let $B$ be the unit disk. Then for $N$ sufficiently large, $T(P_N)<T(B)$. In fact, $T(P_N)$ is a strictly increasing function with respect to $N$ and $\lim_{N\rightarrow \infty}T(P_N)=T(B)$.
\end{theorem}

This theorem will be proved in section 3. The main idea is motivated from Keady-McNabb \cite{KM}. Unlike classical torsion function on regular polygons having no elementary expressions, the Neumann torsion function with vanishing boundary mean on regular polygons, has a simple geometric expression, from which we obtained the monotonicity with respect to the number of sides of regular polygons.

\vskip 0.2cm

\textbf{Remark}:   Only after the work was almost completed, we realized that evaluating $T(\cdot)$ at regular polygons and even tangential polygons have also been studied by Professor Keady in \cite{Keady}, where the readers can find several very nice formulas. 

\vskip 0.3cm

\subsubsection{Instability of Serrin's over-determined System under Dirichlet perturbation}

On the other hand, the system \eqref{EL-N}, as the Euler-Lagrange equation of the extremum function in \eqref{T-function}, is reminiscent of the Serrin's overdetermined system:
 \begin{equation}\label{serrin}
    \begin{cases}
     -\Delta u =1 \qquad &\text{in }\Omega,\\
     \frac{\partial u }{\partial \nu}  =c \qquad &\text{on }\partial\Omega,\\
     u   =0 \qquad &\mbox{on $\partial \Omega$}.
  \end{cases}
\end{equation}
Indeed, solution to $\eqref{serrin}_1$ and $\eqref{serrin}_3$ is the classica Dirichlet torsion function, and the extra condition $\eqref{serrin}_2$ is the due to the stationarity of domain under variation, if $\Omega$ is a stationary shape to $T_\infty(\cdot)$. Coincidentally, $\eqref{serrin}_2$ is incorporated into \eqref{EL-N}. This simple fact leads us to obtain some unexpected observation.

Serrin proved in his seminal paper \cite{Serrin} that the overdetermined system \eqref{serrin} with $\Omega\in \mathcal{C}^2$ admits a solution if and only if $\Omega$ is a ball, and alternative proofs can be found in \cite{BNST} and \cite{Wein}. On the stability of \eqref{serrin}, under the assumption that $\partial \Omega$ is $C^2$, Magnanini and  Poggesi substantiate in \cite{MP-1,MP-2,MP-3}  that if $ \frac{\partial u }{\partial \nu}$ is close to a constant, then $\Omega$ is close to a ball in some appropriate sense.    The following question is natural:

{\bf  Question 1:} If there is a solution  $u$ such that
  \begin{equation}\label{EL-epsilon}
    \begin{cases}
     -\Delta u =1 \qquad &\text{in }\Omega,\\
     \frac{\partial u }{\partial \nu}  =c \qquad &\text{on }\partial\Omega,\\
       \big| \underset{\partial\Omega}{\text{osc}} u \big| <\varepsilon,
  \end{cases}
\end{equation}
  then whether $\Omega$ is close to a ball as $\varepsilon\rightarrow 0$?

Surprisingly, the answer is negative for $n\ge 3$. It is through studying \eqref{T-function} and the associated equation \eqref{EL-N}, we obtain a  counterexample to stability of  Serrin type overdetermined system \eqref{EL-N}. The counterexample will be given in section 5. Whether or not the answer to Question \eqref{EL-epsilon} is positive in dimension $2$ is open to us.

\vskip 0.3cm

\textbf{Remark}:  Even though we have shown that ball is a stationary shape to $T(\cdot)$ given by \eqref{T-function}, we have not been able to classify all stationary shapes to \eqref{T-function}. That is, the rigidity of \eqref{EL-N} is not known to us. Nevertheless, we can show that any annulus cannot be stationary, see also in section 5. There we will also show that prescribing volume, $T(\cdot)$ is always bigger on annulus than on balls.

\section{Proof of Theorem \ref{the:local}}

We first give some notations. Let $F_t(x):=F(t,x)$ be the flow map generated by a smooth
vector field $\eta$ preserving volume. Then we denote $\Omega_t=F_t(\Omega)$, and let $u(t)$ be the unique function on $\Omega_t$ such that
 \begin{equation*}
    T_\beta(\Omega_t)= \frac12 \int_{\Omega} |\nabla u|^2dx+\frac\beta2\int_{\partial \Omega}u^2d\sigma-\int_\Omega fudx.
 \end{equation*}

Noting that   $F_t$ preserves the volume,   one has
\begin{equation}\label{zeta}
   \int_{\Omega_t}\text{div} \eta \, dx=0,\quad     \int_{\Omega_t}\text{div}( (\text{div} \eta ) \eta)\, dx=0,
\end{equation}
since the following formula holds:
\begin{equation*}
  \frac{d}{dt}\int_{\Omega_t}g(x,t)\, dx=\int_{\Omega_t}g_t(x,t)\, dx+\int_{\partial \Omega_t}g(x,t) \eta\cdot \nu(t) \, d\sigma_t.
\end{equation*}
By  the divergence theorem and equations \eqref{EL-Tb}, we have
\begin{equation}\label{c}
c= \frac1{P(\Omega)}\int_{\partial \Omega}\frac{\partial u_\Omega}{\partial \nu}\, d\sigma=  \frac1{P(\Omega)} \int_{\Omega}\Delta  u_\Omega \, dx=-\frac{\int_{\Omega} f\, dx}{P(\Omega)}=-\frac{\bar{f} |\Omega|}{P(\Omega)},
\end{equation}
where $\bar{f}=\frac1{|\Omega|}\int_{\Omega} fdx$.
We next state some evolution equations of some geometric quantities, which have been proven in  \cite[Proposition 2.1]{Huang-Li-Li}.
\begin{proposition}\label{pro:formula}
 Let $F_t(x):=F(t,x)$ be the flow map generated by a smooth vector field $\eta$,   $M_t=F_t(M)$,
$\sigma_t$  be the volume element of $M_t$, $\nu(t)$ be the unit normal field along $M_t$ and $h(t)$ be the second fundamental form of $M_t$, then we have
\begin{equation}\label{sigma-dt}
   \frac{d}{dt}d\sigma_t=(\text{div}_{M_t}\eta)d\sigma_t,
\end{equation}
\begin{equation}\label{zeta-dt}
    \frac{d}{dt}(\eta(F_t)\cdot \nu(t))=(\eta(F_t)\cdot \nu(t))(\text{div}\eta-\text{div}_{M_t}\eta)\circ F_t,
\end{equation}
and
\begin{equation}
   h_{ij}'(t)=-\langle\nabla_i\nabla_j \eta,\nu(t)\rangle,
\end{equation}
where $h_{ij}(t)$ and  $\nabla_i\nabla_j \eta$ are the $i,j$-components of $h(t)$ and the Hessian of $\eta$ on $M_t$, respectively, under local coordinates of $M_t$.

If $M$ is an $(n -1)$-sphere of radius $R$, then we also have
\begin{equation}\label{H-dt}
  \frac{d}{dt} \big|_{t=0}  H
      =    -\Delta_{M}(\eta \cdot \nu)-\frac{n-1}{R^2}\eta \cdot \nu.
\end{equation}

\end{proposition}

Now we first calculate  the first variation of energy function $T_\beta$.
 \begin{lemma}\label{lem:first-var}
Let $f>0$  be a smooth  function. Then for $|t|$ small,  we find
 \begin{equation}\label{first-var}
    \frac{d}{dt}T_\beta(\Omega_t)=\int_{\partial\Omega_t}\left[  \frac12  |\nabla  u(t)|^2+2c\beta u(t)-cH(t)u(t)+\frac{\beta}2H(t)u^2(t)- \beta^2u^2(t)-fu(t)\right]\eta\cdot \nu(t)\, d\sigma_t,
 \end{equation}
 where    $H(t)$ is the mean curvature of the boundary  $\partial\Omega_t$
 and $\sigma_t$ is the volume element for $\partial\Omega_t$.
 \end{lemma}
 \begin{proof}
 We first define $u'(t)$ by 
    \begin{equation*}
      u'(t)(F_t(x))=\frac d{dt}\left(u(t)(F_t(x))\right)-\nabla u(t)(F_t(x))\cdot \eta(F_t(x)).
    \end{equation*}
Similar to  \cite[Proposition 3.1]{Huang-Li-Li} and \cite[Proposition 3.1]{Li-Yang},  we have
    \begin{equation*}
  \begin{aligned}
        \frac d{dt}T_\beta(\Omega_t)
        &= \int_{\Omega_t}fu'(t)\, dx+\int_{\partial\Omega_t}\left(\frac{\partial u(t)}{\partial \nu} u'(t)+ \frac12  |\nabla  u(t)|^2\eta\cdot \nu(t)\right)\, d\sigma_t\\
        &\ \ +\frac\beta2\int_{\partial\Omega_t}\left( 2  u(t) u'(t)+2 \frac{\partial u(t)}{\partial \nu} u(t)\eta\cdot \nu(t)+u^2(t)H(t)\eta\cdot \nu(t) \right)\, d\sigma_t\\
        &\ \ -\int_{\Omega_t}fu'(t)\, dx-\int_{\partial\Omega_t}   fu(t)\eta\cdot \nu(t)\, d\sigma_t\\
        &=c\int_{\partial\Omega_t}  u'(t) d\sigma_t+\int_{\partial\Omega_t}\left(  \frac12  |\nabla  u(t)|^2+\beta u(t)\frac{\partial u(t)}{\partial \nu} +\frac\beta2 u^2(t)H(t)- fu(t)\right)\eta\cdot \nu(t)\, d\sigma_t.
  \end{aligned}
    \end{equation*}

  On the other hand,   by the condition that $\int_{\partial\Omega_t} u(t) \,d\sigma_t=0$,
taking derivative yields
    \begin{equation}\label{u'}
     \int_{\partial\Omega_t} u'(t)\, d\sigma_t=  -\int_{\partial\Omega_t}\left(\frac{\partial u(t)}{\partial \nu}+u(t)H(t)\right)\eta\cdot \nu(t) \,d\sigma_t.
    \end{equation}
   
  It follows from \eqref{EL-Tb}$_2$ and  \eqref{zeta}  that
    \begin{equation*}
    \begin{aligned}
    &\quad c\int_{\partial\Omega_t}  u'(t) 
    \,d\sigma_t+    \beta\int_{\partial\Omega_t} u(t)\frac{\partial u(t)}{\partial \nu}\, d\sigma_t\\
   &=\int_{\partial\Omega_t} \left[-c\left(\frac{\partial u(t)}{\partial \nu}+u(t)H(t)\right)+\beta u(t) \left(c-\beta u(t)\right) \right]\eta\cdot \nu(t) \, d\sigma_t\\
   &=\int_{\partial\Omega_t} \left[-c^2+2c\beta u(t)  -cu(t)H(t)  -\beta^2 u^2(t)\right]\eta\cdot \nu(t)\, d\sigma_t\\
   &=\int_{\partial\Omega_t} \left[ 2c\beta u(t)  -cu(t)H(t)  -\beta^2 u^2(t)\right]\eta\cdot \nu(t)\, d\sigma_t.
    \end{aligned}
    \end{equation*}
Putting these results together,  we conclude  that
\begin{equation*}
   \frac{d}{dt}T_\beta(\Omega_t)=\int_{\partial\Omega_t}\left[  \frac12  |\nabla  u(t)|^2+2c\beta u(t)-cH(t)u(t)+ \frac{\beta}2H(t)u^2(t)- \beta^2u^2(t)-fu(t)\right]\eta\cdot \nu(t)\,d\sigma_t
\end{equation*}
as desired.
 \end{proof}
\begin{corollary}\label{ball-station}
Under the assumptions of  Lemma \ref{lem:first-var},   $\Omega$ is a stationary shape to  $T_\beta(\cdot)$ if and only if there exists a solution to the following
system
\begin{equation}\label{station-eq}
  \begin{cases}
     -\Delta u=f \qquad &\text{in }\Omega,\\
      \frac{\partial u}{\partial \nu}+\beta u=c \qquad &\text{on }\partial\Omega,\\
            \frac12  |\nabla  u|^2+2c\beta u-cHu+ \frac\beta2 H u^2- \beta^2u^2-fu=constant     &\text{on }\partial\Omega,\\
              \int_{\partial \Omega}u 
              \, d\sigma_t=0.
  \end{cases}
\end{equation}
In particular, any ball is stationary for radial function $f$.
\end{corollary}

 Let    $B_R\subset \R^n$ be a ball of radius $R$ centered at the origin. Then
with the help of formulas in  Proposition \ref{pro:formula},   we now estimate the second variation of $T_\beta(B_R)$.
\begin{lemma}\label{lem:second-var}
  Let $\Omega_0=B_R$, $v=u'(0)$ and   $f>0$  be a smooth  radial function. Then we obtain
 \begin{equation}\label{second-var}
    \frac{d^2}{dt^2} \bigg|_{t=0}T_\beta(\Omega_t)= \int_{\partial B_R}\left(v\zeta+u_r\zeta^2\right) \left(u_{rr}+c\beta\right) \, d\sigma,
 \end{equation}
 where $\zeta=\eta\cdot \nu$ and $\nu$   is the unit  outer normal to boundary  $\partial B_R$.
 \end{lemma}
 \begin{proof}
 Performing the similar procedure of  \cite[Proposition 3.3]{Li-Yang},   we also get
 \begin{equation}\label{v-eq}
      \begin{cases}
     \Delta v=0,\qquad &\text{in }B_R;\\
     \frac{\partial v}{\partial \nu} =- u_{rr}\zeta-\beta u_r\zeta-\beta v,\qquad &\text{on }\partial B_R.\\
  \end{cases}
 \end{equation}

 Now we  calculate the second shape derivative  based on \eqref{first-var}. To this end,     we  derive from \eqref{zeta},   \eqref{sigma-dt} and \eqref{zeta-dt} that
 \begin{equation*}
   \begin{aligned}
 I_1&:= \frac{d}{dt}\Big |_{t=0}\int_{\partial\Omega_t}\frac12  |\nabla  u(t)|^2\eta\cdot \nu(t)\,d\sigma_t\\
      &=\int_{\partial B_R}  \left[\zeta\nabla  u\cdot \nabla v+\langle \nabla  u,\nabla^2  u \eta\rangle \zeta+\frac12  |\nabla  u(t)|^2 \zeta\left(\text {div}\eta-\text {div}_{\partial B_R}\eta\right)+\frac12  |\nabla  u(t)|^2 \zeta \text {div}_{\partial B_R}\eta \right]\, d\sigma\\
      &=\int_{\partial B_R}  \left[\zeta  u_r   v_r+   u_r  u_{rr}\zeta^2+\frac12 u_r^2 \zeta \text {div}\eta  \right] \,d\sigma\\
      &=\int_{\partial B_R}  \left(\zeta  u_r   v_r+   u_r  u_{rr}\zeta^2  \right)\, d\sigma,
      \end{aligned}
 \end{equation*}
 where we used  that  $u$  is radial and
 \begin{equation*}
  \langle \nabla  u,\nabla^2  u \eta\rangle=u_r \nabla^2u:\eta\otimes\nu=u_r  u_{rr}\zeta.
 \end{equation*}
 Using  \eqref{zeta}--\eqref{zeta-dt} and \eqref{H-dt},  we get
  \begin{equation*}
   \begin{aligned}
 I_2&:= \frac{d}{dt}\Big |_{t=0}\int_{\partial\Omega_t}c \left( 2 \beta -H(t)\right) u(t)\eta\cdot \nu(t)\,d\sigma_t \\
 &=c\int_{\partial B_R}\left[  u\zeta\left(  \Delta_{\partial B_R}\zeta +\frac{n-1}{R^2}\zeta\right)+\left( 2 \beta  - H  \right)\left(v\zeta+u_r\zeta^2\right)  +\left( 2 \beta  - H  \right)u\zeta\text {div}\eta  \right]\, d\sigma\\
      &=c\int_{\partial B_R}\left[ u\zeta\left(  \Delta_{\partial B_R}\zeta +\frac{n-1}{R^2}\zeta\right) +\left( 2 \beta  - H  \right)\left(v\zeta+u_r\zeta^2\right) \right]\, d\sigma
      \end{aligned}
 \end{equation*}
and
  \begin{equation*}
   \begin{aligned}
 I_3&:= \frac{d}{dt}\Big |_{t=0}\int_{\partial\Omega_t} \frac{\beta}2H(t)u^2(t) \eta\cdot \nu(t)\, d\sigma_t \\
 &=\frac{\beta}2\int_{\partial B_R}\left[u^2\zeta\left(  -\Delta_{\partial B_R}\zeta -\frac{n-1}{R^2}\zeta\right)+2Hu\left(v\zeta+u_r\zeta^2\right)+ u^2H\zeta\text {div}\eta\right]\, d\sigma \\
      &=\frac{\beta}2\int_{\partial B_R}\left[-u^2\zeta\left( \Delta_{\partial B_R}\zeta+\frac{n-1}{R^2}\zeta\right)+2Hu\left(v\zeta+u_r\zeta^2\right)\right]\, d\sigma
      \end{aligned}
 \end{equation*}
 Similarly, one has
   \begin{equation*}
     \begin{aligned}
       I_4&:= \frac{d}{dt}\Big |_{t=0}\int_{\partial\Omega_t} - \left(\beta^2u^2(t)+fu(t)\right)\eta\cdot \nu(t)\,d\sigma_t\\
       &=  -\int_{\partial B_R}\left[2\beta u\zeta \left(v\zeta+u_r\zeta^2\right)+
     f_ru\zeta^2+f\left(v\zeta+u_r\zeta^2\right)  + \left(\beta^2u^2(t)+fu(t)\right)\zeta\text {div}\eta \right]\, d\sigma\\
     &= -\int_{\partial B_R}\left[2\beta u\zeta \left(v\zeta+u_r\zeta^2\right)+
     f_ru\zeta^2+f\left(v\zeta+u_r\zeta^2\right)   \right]\, d\sigma.
     \end{aligned}
 \end{equation*}

 On the other hand,   on $\partial B_R$,  it follows from  \eqref{c} and the system \eqref{EL-Tb}  that
 \begin{equation}\label{u(R)}
   0= \int_{\partial B_R} u\,d\sigma= u(R) P (B_R)\Rightarrow u(R)=0;
 \end{equation}
 \begin{equation}\label{ur(R)}
   u_r(R)=c-\beta u(R)=c=-\frac{ R}{n}\bar f_{B_R}, \quad \text  { where } \bar f_{B_R}=\frac1{|B_R|}\int_{B_R} fdx;
 \end{equation}
 \begin{equation}\label{urr(R)}
   u_{rr}(R)=-f(R)-\frac{n-1}{R}u_r=-f(R)+\frac{n-1}{n}\bar f_{B_R};
 \end{equation}
since  $u$ and $f$ are radial.
 Then putting the above equations together, we conclude  from \eqref{v-eq}$_2$ and \eqref{u(R)}--\eqref{urr(R)}  that
 \begin{equation*}
 \begin{aligned}
           \frac{d^2}{dt^2} \bigg|_{t=0}T_\beta(\Omega_t)&=\sum_{i=1}^4 I_i\\
           &=\int_{\partial B_R}\left[  u_r\zeta \left(  v_r+    u_{rr}\zeta \right)+ \left(v\zeta+u_r\zeta^2\right)\left(  2 c\beta  - cH -f\right) \right]  \,d\sigma\\
           &= \int_{\partial B_R}\left(v\zeta+u_r\zeta^2\right) \left(-\beta u_r +2c\beta+u_{rr}\right) \,d\sigma\\
           &= \int_{\partial B_R}\left(v\zeta+u_r\zeta^2\right) \left(  c\beta+u_{rr}\right) \,  d\sigma.
 \end{aligned}
\end{equation*}

 \end{proof}
Combining Lemma \ref{lem:first-var} and Lemma \ref{lem:second-var}, we can prove Theorem \ref{the:local}.

\begin{proof}[Proof of Theorem \ref{the:local}]
The  stationary result follows from Corollary \ref{ball-station}.

 Let $u=u_{B_R}$, $v$ satisfy \eqref{v-eq} and $\eta$ be a smooth velocity field of the volume preserving flow starting from $B_R$. Then from \eqref{ur(R)} and \eqref{urr(R)},  we  see that on $\partial B_R$,
 \begin{equation}\label{F(R)}
 \begin{aligned}
       F(R)&:=-c\beta-u_{rr}(R)=\frac{\beta R}{n}\bar f_{B_R}+f(R)-\frac{n-1}{n}\bar f_{B_R}\\
       &=f(R)-\frac{n-1-\beta R}{n}\bar f_{B_R}\ge 0,
 \end{aligned}
 \end{equation}
  where the condition \eqref{f-condition} implies the last inequality.
From \eqref{zeta} and \eqref{u'} with $t=0$,  it is readily checked that
 \begin{equation*}
  \int_{\partial B_R} v \, d\sigma=0.
 \end{equation*}
 Since the second steklov eigenvalue on   $B_R$ is $\frac1{R}$ and
the equations \eqref{v-eq}, we further obtain
 \begin{equation*}
 \begin{aligned}
      \int_{\partial B_R} v^2 \, d\sigma&\le R  \int_{  B_R}|\nabla v|^2\, dx=R  \int_{\partial B_R}\frac{\partial v}{\partial \nu}v\, d\sigma=-R\int_{\partial B_R} \left( u_{rr}\zeta+\beta u_r\zeta+\beta v\right)v \, d\sigma\\
      &=-\beta R \int_{\partial B_R} v^2 \, d\sigma- R\int_{\partial B_R} \left( u_{rr} +\beta u_r \right)v \zeta\, d\sigma\\
      &=-\beta R \int_{\partial B_R} v^2 \,d\sigma+R\int_{\partial B_R}F(R)v \zeta \,d\sigma.
 \end{aligned}
 \end{equation*}
 Therefore,
 \begin{equation}\label{Fv-lowbd}
 \int_{\partial B_R} v^2\, d\sigma\le  \frac{R}{1+\beta R}\int_{\partial B_R}F(R)v \zeta \,d\sigma,
 \end{equation}
which together with Cauchy-Schwarz  inequality implies
 \begin{equation*}
 \begin{aligned}
  \left( \int_{\partial B_R} F(R)  v\zeta \,d\sigma\right)^2&\le  \int_{\partial B_R} F(R)^2\zeta^2\, d\sigma  \int_{\partial B_R}  v^2\,d\sigma\\
  &\le  \frac{R}{1+\beta R}\int_{\partial B_R}F(R)v \zeta \,d\sigma\int_{\partial B_R} F(R)^2 \zeta^2 \,d\sigma.
  \end{aligned}
 \end{equation*}
 So we find
\begin{equation}\label{Fv-upbd}
   \int_{\partial B_R}  F(R) v\zeta \,d\sigma\le  \frac{R}{1+\beta R}\int_{\partial B_R}F(R)^2 \zeta^2 \,d\sigma
\end{equation}
due to   \eqref{Fv-lowbd}.

\eqref{second-var}, \eqref{F(R)} and \eqref{Fv-upbd}  indicate that
  \begin{equation*}
\begin{aligned}
           \frac{d^2}{dt^2} \bigg|_{t=0}T_\beta(F_t(B_R))&=-F(R) \int_{\partial B_R}\left(v\zeta+u_r\zeta^2\right)  \,d\sigma\\
           &\ge -\frac{RF(R)^2}{1+\beta R}\int_{\partial B_R}\zeta^2 d\sigma- F(R)u_r(R)\int_{\partial B_R} \zeta^2 \, d\sigma\\
           &=\frac{R}{1+\beta R} F(R)\left(\bar f_{B_R}-f(R)\right) \int_{\partial B_R} \zeta^2 \, d\sigma\\
           &=\frac{R}{1+\beta R}\left(f(R)-\frac{n-1-\beta R}{n}\bar f_{B_R}\right)\left(\bar f_{B_R}-f(R)\right)\int_{\partial B_R} \zeta^2 \, d\sigma\\
           &\ge0.
\end{aligned}
  \end{equation*}
Hence we have shown that \eqref{f-condition} is a sufficient condition to guarantee that ball is a stable shape to $T_\beta(\cdot)$ under smooth volume-preserving perturbation.

Next, we show the necessity of \eqref{f-condition}. The idea is to choose $F_t$ to be the translation map with constant speed, and hence $\zeta$ is a linear combination of coordinate functions restricted on $\partial B_R$.  By solving \eqref{v-eq} for such $\zeta$, $v$ is also a linear combination of coordinate functions. Hence $v$ is exactly a second eigenfunction of Steklov eigenvalue, and hence the inequality in \eqref{Fv-lowbd} becomes equality. By the equality condition of Schwarz inequality, the inequality in \eqref{Fv-upbd} also becomes equality. Therefore, previous computation yields
\begin{align*}
      \frac{d^2}{dt^2} \bigg|_{t=0}T_\beta(F_t(B_R))=\frac{R}{1+\beta R}\left(f(R)-\frac{n-1-\beta R}{n}\bar f_{B_R}\right)\left(\bar f_{B_R}-f(R)\right)\int_{\partial B_R} \zeta^2 \, d\sigma.
\end{align*}
In order for the above to be nonnegative, if forces $f$ to satisfy \eqref{f-condition}.
\end{proof}

\section{Instability of $T(\cdot)$ under Lipschitz perturbation: regular polygon example}
In this section, we will prove Theorem \ref{PN}.

\begin{proof}
    Let $u$ be the function where the infimum in \eqref{T-function} is achieved. Let $\Omega$ be a regular polygon centered at the origin, then on the boundary, $x \cdot \nu =\rho$, where $\rho:=\rho_\Omega$ is the inradius of $\Omega$, the radius of the inscribed circle. Hence we have the following explicit formula for $u$:
\begin{align}
    \label{formulaforu}
u(x)=\frac{1}{4P(\Omega)}\int_{\partial \Omega}|x|^2 \, ds-\frac{1}{4}|x|^2.    
\end{align}
Note that 
\begin{align*}
    \int_\Omega \Delta (|x|^4)\, dx=16\int_{\Omega}|x|^2 \, dx.
\end{align*}
Also, by divergence theorem, the left hand side above is also equal to 
\begin{align*}
    \int_{\partial \Omega} 4|x|^2 (x\cdot \nu)\, ds=4\rho \int_{\partial \Omega}|x|^2.
\end{align*}
Hence 
\begin{align}
    \label{zilv1}
\int_\Omega |x|^2 \, dx=\frac{\rho}{4}\int_{\partial \Omega}|x|^2 \, ds.
\end{align}
Note that 
\begin{align}
    T(\Omega)=-\frac{1}{2} E(\Omega), 
\end{align}where 
\begin{align}
    \label{defofE}
E(\Omega):=\int_\Omega u \, dx,
\end{align}hence to prove that $T(P_N)$ is strictly increasing, is equivalent to proving $E(P_N)$ is strictly decreasing. 

From \eqref{formulaforu} and \eqref{zilv1}, and since $\rho P(\Omega)=2|\Omega|=2\pi$, we have that for a regular polygon $\Omega$, 
\begin{align*}
    E(\Omega):=\frac{|\Omega|}{4P(\Omega)}\int_{\partial \Omega}|x|^2 \, ds-\frac{\rho}{16}\int_{\partial \Omega}|x|^2 \, ds=\frac{\rho}{16}\int_{\partial \Omega}|x|^2 \, ds.
\end{align*}

Considering one side $\Gamma$ of $\Omega=P_N$, and let the horizontal axis $\xi$ pass through $\Gamma$, and let the origin be the center of $P_N$. Let $L$ be the length of $\Gamma$, and hence $L=2\rho \tan(\pi/N)$. Then 
\begin{align*}
    \int_\Omega |x|^2 \, ds=N \int_\Gamma |x|^2 \, ds=&N \rho^2 L+N \int_{-L/2}^{L/2} \xi^2 \, d\xi\\
    =& 2N\rho^3 \tan(\pi/N)+\frac{2}{3}N\rho^3\tan^3(\pi/N).
\end{align*}
Hence 
\begin{align}
    E(P_N)=\frac{\rho^4 N \tan (\pi/N)}{24}(3+\tan^2(\pi/N)).
\end{align}
Since $N\rho^2\tan(\pi/N)=|\Omega|=\pi$, we have
\begin{align}
\label{final}
    E(P_N)=\frac{\pi^2 (3+\tan^2(\pi/N))}{24N\tan(\pi/N)}.
\end{align}
Baby calculus implies that as $N \rightarrow \infty$,
\begin{align*}
    E(P_N)=\frac{\pi}{24}\left(3+\frac{2}{3}\frac{1}{N^4}+O(\frac{1}{N^6})\right).
\end{align*}
Hence when $N$ is large, $E(P_N)$ is strictly decreasing, and $T(P_N)$ is strictly increasing, and hence ball is not a local minimizer under Lipschitz variation.

Some further work can directly show that $E(P_N)$ is strictly decreasing when $N\ge 3$ is increasing, and hence $T(P_N)$ is strictly increasing for all $N \ge 3$.
\end{proof}

\section{Discussion between $T_\beta(\Omega)$ and $J_\beta(\Omega)$}\label{TJ}

In this section, we mainly consider the special case that $f\equiv1$ in   energy functionals  $T_\beta(\Omega)$ and $J_\beta(\Omega)$, which are defined by \eqref{T-beta} and \eqref{J-beta}, respectively.
  We know that the Euler-Lagrange  equation in $T_\beta(\Omega)$ is
\begin{equation}\label{EL-Tb-1}
    \begin{cases}
     -\Delta u=1 \qquad &\text{in }\Omega;\\
     \frac{\partial u}{\partial \nu}+\beta u  =c \qquad &\text{on }\partial\Omega;\\
     \int_{\partial \Omega}ud\sigma=0,
  \end{cases}
\end{equation}
which has a unique solution. If the infimum in the definition of $J_\beta(\Omega)$ is attained at $\hat u_\Omega$, then
\begin{equation*}
  u_\Omega:=\hat u_\Omega-\frac1{P(\Omega)} \int_{\partial \Omega}\hat u_\Omega d\sigma
\end{equation*}
automatically solves \eqref{EL-Tb-1}, i.e., $ u_\Omega$ is the function where the infimum in \eqref{T-beta} is attained.  Note that  $\hat u_\Omega$ satisfies \eqref{EL-R}. We deduce from \eqref{c} and \eqref{EL-Tb-1}$_2$ that
\begin{equation*}
  -|\Omega|=cP(\Omega)=\int_{\partial \Omega}  \left( \frac{\partial u_\Omega}{\partial \nu}+\beta u_\Omega\right) d\sigma=-\beta  \int_{\partial \Omega}\hat u_\Omega d\sigma,
\end{equation*}
which implies that
\begin{equation*}
  u_\Omega=\hat u_\Omega-\frac{|\Omega|}{\beta P(\Omega)}.
\end{equation*}
Hence, we have
\begin{equation}\label{TJ+omega}
  T_\beta(\Omega)=-\frac12\int_{  \Omega} \left( \hat u_\Omega-\frac{|\Omega|}{\beta P(\Omega)}\right) dx=  J_\beta(\Omega)+\frac{|\Omega|^2}{2\beta P(\Omega)},
\end{equation}
where we use the equality $J_\beta(\Omega)=-\frac12\int_{  \Omega}   \hat u_\Omega dx$.

In the following,   we  will drop the symbol $\Omega$ of  $u_\Omega$  and $\hat u_\Omega$.  Clearly, the ball $B_R$ is critical to $T_\beta(\cdot)$ under volume preserving flow. Also,
  \begin{equation*}
\begin{aligned}
\frac{d^2}{dt^2} \bigg|_{t=0}T_\beta(F_t(B_R))
&=\frac{d^2}{dt^2} \bigg|_{t=0}J_\beta(F_t(B_R))+\frac{|B_R|^2}{2\beta}\frac{d^2}{dt^2} \bigg|_{t=0}\frac{1}{P(F_t(B_R))} \\
&= \frac\beta2 \hat u^2(R) \int_{\partial B_R} \left(- \Delta_{\partial B_R}\zeta-\frac{n-1}{R^2}\zeta\right)\zeta  d\sigma
+\int_{\partial B_R}\left(\hat v\zeta+\hat u_r\zeta^2\right)
\left( \hat u_{rr}-\beta^2\hat u \right) d\sigma\\
           &\ \ +\frac{|B_R|^2}{2\beta P^2(B_R)}\int_{\partial B_R} \left( \Delta_{\partial B_R}\zeta+\frac{n-1}{R^2}\zeta\right)\zeta  d\sigma\\
  &:=I+II+III,
\end{aligned}
  \end{equation*}
  which follows from \cite{Li-Yang} and the following
  \begin{equation*}
  \begin{aligned}
      \frac{d^2}{dt^2} \bigg|_{t=0}\frac{1}{P(F_t(B_R))}&=\frac{d}{dt} \bigg|_{t=0}\left[-P(F_t(B_R))^{-2} \int_{\partial F_t(B_R)} H(t)\eta\cdot \nu(t)d\sigma_t\right]\\
      &=2P(B_R)^{-3} \left( \int_{\partial   B_R } H\zeta d\sigma\right)^2- P(B_R)^{-2}\int_{\partial B_R} \left( -\Delta_{\partial B_R}\zeta-\frac{n-1}{R^2}\zeta\right)\zeta  d\sigma\\
      &\ \  -P(B_R)^{-2}\int_{\partial B_R} H\zeta \text{div}\eta  d\sigma\\
      &=P(B_R)^{-2}\int_{\partial B_R} \left( \Delta_{\partial B_R}\zeta+\frac{n-1}{R^2}\zeta\right)\zeta  d\sigma 
  \end{aligned}
  \end{equation*}
owing to \eqref{zeta} and \eqref{sigma-dt}--\eqref{H-dt}.
We can check that  
\begin{equation*}
  I+III=0.
\end{equation*}
 Indeed,    by \eqref{EL-R}  and divergence theorem, we have
\begin{equation*}
  \beta \hat u(R)\int_{\partial B_R}d \sigma=\int_{\partial B_R} \beta  \hat ud \sigma
  =-\int_{\partial B_R}  \frac{\partial\hat u}{\partial \nu} d \sigma
  =-\int_{  B_R}  \Delta \hat u  d \sigma=\int_{  B_R}     d \sigma,
\end{equation*}
which evinces that
\begin{equation*}
  \hat u(R)=\frac{R}{n\beta}.
\end{equation*}
So we have
\begin{equation*}
  \frac\beta2 \hat u^2(R)=\frac{R^2}{2\beta n^2}=\frac{|B_R|^2}{2\beta P^2(B_R)}.
\end{equation*}
Hence $I+III=0$. 

Furthermore, by \cite{Li-Yang}, when $f \equiv 1$, $II\ge 0$. Hence 
\begin{equation*}
  \frac{d^2}{dt^2} \bigg|_{t=0}T_\beta(F_t(B_R))\ge 0,
\end{equation*}
This matches Theorem \ref{the:local}.

Since  the first eigenvalue of Laplacian on $\partial B_R$ is $\frac{n-1}{R^2}$, essentially we proved the following.
\begin{corollary}
  Let $J_\beta(\Omega)$ be defined as \eqref{J-beta}, where $f\equiv1$. Define  
  \begin{equation*}
    J^\lambda_\beta(\Omega):=J_\beta(\Omega)+\lambda \frac{|\Omega|^2}{  P(\Omega)}, \quad\mbox{ for }\beta>0.
  \end{equation*}
   Then under smooth volume-preserving perturbation, ball is stable for  $ J_\beta^\lambda(\cdot)$ whenever $\lambda\le \frac1{2\beta}$. When $\lambda=\frac1{2\beta}$, $J_\beta^\lambda(\cdot)=T_\beta(\cdot)$.
\end{corollary}


\section{Examples}

In this section, first, we provide a counterexample to instability of Serrin's system for Dirichlet perturbation.

{\bf  Counterexample:} Let $\Omega=\prod_{i=1}^n(-a_i,a_i),\ a_1\le a_2\le\cdots\le a_n$ and if $u$ is a solution to \eqref{EL-epsilon}, then up to a constant,
\begin{equation*}
  u=-\left(\sum_{i=1}^n\frac{x_i^2}{a_i}\right)\left(\sum_{i=1}^n\frac{2}{a_i}\right)^{-1}
  =-\frac{\sigma_n}{2\sigma_{n-1}}\sum_{i=1}^n\frac{x_i^2}{a_i},
\end{equation*}
where $\sigma_k$ is the $k$-th elementary symmetric polynomial in $[a_1, a_2,\cdots, a_n]$. Let $a_1=\varepsilon^{n-1}$ and $a_i=\frac1\varepsilon$ for $i\ge 2$. Straightforward calculation implies that
\begin{equation*}
   \underset{\partial\Omega}{\text{osc}} u\le  \underset{\overline\Omega}{\text{osc}} u=\frac{\sigma_1\sigma_n}{2\sigma_{n-1}}\begin{cases}
     =\frac{1}{2}|\Omega| \qquad &\text{ for }n=2,\\
     \rightarrow 0,  \text{ as } \varepsilon\rightarrow 0   &\text{ for }n\ge 3,
   \end{cases}
\end{equation*}
since  when $n\ge 3$, 
\begin{equation*}
  \sigma_{1} \sim \frac{n-1}{\varepsilon},\quad \sigma_{n-1}\sim \frac{1}{\varepsilon^{n-1}},
   \quad \text{  and so  }\frac{\sigma_1 }{2\sigma_{n-1}}\sim \varepsilon^{n-2}\rightarrow0,  \text{ as } \varepsilon\rightarrow 0.
\end{equation*}
 Note that when $n=2$,
 \begin{equation*}
   \underset{\partial\Omega}{\text{osc}} u=u(a_1,0)-u(a_1,a_2)= \frac{a_1a_2^2}{2(a_1+a_2)}=\frac{|\Omega|}{8}\frac{a_2}{a_1+a_2}\ge \frac{|\Omega|}{16}.
 \end{equation*}
It leaves open that whether the answer to  Question 1 is positive for $n=2$.

\vskip 0.3cm

Next, we consider whether some annulus in $\R^n$ can be a stationary shape to  $T(\cdot)$ introduced by \eqref{T-function}.

\textbf{Example: Annulus }

\vskip 2mm
\par

Without loss of generality,  we define the annulus  $\Omega=\{x\in\mathbb R^n:1<|x|<b\}$, then the solution $u$ of \eqref{station-eq} with  $\beta=0$ and $f\equiv1$ satisfies
\begin{equation}\label{sol-annulus}
  u(x)=\begin{cases}
   a_1|x|^2+a_2 |x|^{2-n}+a_3, \qquad &n\ge 3;\\
   a_1|x|^2+a_2\log|x|+a_3,  &n=2.
  \end{cases}
\end{equation}
 From \eqref{station-eq}$_1$  we see that
 \begin{equation*}
   \frac1{r^{n-1}}(r^{n-1}u_{r})_r=-1,
 \end{equation*}
 which implies that
 \begin{equation*}
\begin{cases}
     \frac1{r^{n-1}}\left(2a_1r^n-(n-2)a_2\right)_r=-1, \qquad &n\ge 3\\
     \frac1{r^{n-1}}\left(2a_1r^n+a_2\right)_r=-1,   &n=2\\
\end{cases}
\Rightarrow a_1=-\frac1{2n}.
\end{equation*}
It follows from \eqref{station-eq}$_2$ that
\begin{equation}\label{u-r}
   u_r(b)=-u_r(1)=-\frac{|\Omega|}{P(\Omega)},
\end{equation}
By \eqref{station-eq}$_3$, we find
\begin{equation*}
 \frac12 u_r(b)^2+\left(\frac{|\Omega|}{P(\Omega)}H(b)-1\right)u(b)= \frac12 u_r(1)^2+\left(\frac{|\Omega|}{P(\Omega)}H(1)-1\right)u(1),
\end{equation*}
which together with   \eqref{u-r}  entails that
\begin{equation}\label{u-station-con}
   \frac{u(b)}{u(1)}=-\frac{b\left[(n-1)(b^n-1)+n(b^{n-1}+1)\right]}{(n-1)(b^n-1)-nb(b^{n-1}+1)}
   =\frac{b\left[(n-1)b^n+nb^{n-1}+1\right]}{b^n+nb+n-1}.
\end{equation}

 On the other hand,  the condition that $\int_{\partial \Omega}ud\sigma=0$  gives that
\begin{equation}\label{u-bdy}
 \frac{u(b)}{u(1)}=-\frac{P(B_1(0))}{P(B_b(0))}=-\frac{1}{b^{n-1}}.
\end{equation}
Thus, we  derive  from \eqref{u-station-con} and   \eqref{u-bdy} that
\begin{equation*}
  - \frac{1}{b^{n}}=\frac{ (n-1)b^n+nb^{n-1}+1}{b^n+nb+n-1}.
\end{equation*}
This is impossible since the right hand side of the above equation is positive for $n\ge 2$ and $b>1$.  Therefore, the annulus  in $\R^n$   {\bf can be not a stationary shape} to $T(\cdot)$.
\vskip 3mm

Next, we will compare the energy   $T(\cdot)$ between  the ball   and the annulus  in $\R^2$ owing the same volume.

Let
\begin{equation*}
  \Omega_1=\{x\in\mathbb R^2:1<|x|<b\} \quad \mbox{ and }\quad  \Omega_2=\{x\in\mathbb R^2:|x|<\sqrt{b^2-1}\}
\end{equation*}
 and $u_{i}$ solves \eqref{EL-N} for $\Omega_i,\ i=1,2$. Then a direct calculation shows that
 \begin{equation}\label{sol-1}
   u_1(x)=-\frac{1}{4}|x|^2+a_2\log |x|+a_3
 \end{equation}
 and
 \begin{equation}\label{sol-2}
   u_2(x)=-\frac{|x|^2-R^2}{4},\quad R=\sqrt{b^2-1}.
 \end{equation}
From  \eqref{u-r} and \eqref{u-bdy} we see that
\begin{equation*}
  a_2=\frac b2,\quad a_3=\frac1{b+1}\left(\frac{b^3+1}4-\frac{b^2}2\log b\right).
\end{equation*}
 It is easy to check that  $T(\Omega_i)=-\frac12 \int_{\Omega_i}u_i\, dx$.

 For the ball, we calculate
 \begin{equation}\label{T-ball}
   \int_{\Omega_2}u_2(x)dx=\frac\pi2 \int_0^Rr(R^2-r^2)dr=\frac{\pi R^4}8=\frac{\pi}{8}(b^2-1)^2.
 \end{equation}
Then for the annulus $\Omega_1$, a direct computation shows
\begin{equation}\label{T-annu}
\begin{aligned}
   \int_{\Omega_1}u_1\, dx&= 2\pi \int_1^b\left[-\frac{r^2}{4}+\frac b2 \log r+\frac1{b+1}\left(\frac{b^3+1}4-\frac{b^2}2\log b\right)\right]r\, dr\\
   &=-\frac{\pi(b^4-1)}{8}+\frac{\pi b}{4}\left(2b^2\log b-b^2+1\right)+(b-1)\pi\left(\frac{b^3+1}4-\frac{b^2}2\log b\right).
\end{aligned}
\end{equation}
Hence, we have
\begin{equation*}
\begin{aligned}
\int_{\Omega_2}u_2\, dx-\int_{\Omega_1}u_1\, dx
     =\frac\pi 4\left(2b^3-b^2-2b^2\log b-2b+1\right).
\end{aligned}
\end{equation*}
Let $h(b)=2b^3-b^2-2b^2\log b-2b+1$, then we know that
\begin{equation*}
  h'(b)>0 \quad \mbox{for }b> 1,
\end{equation*}
which implies $ h(b)>h(1)=0$. So one has
\begin{equation*}
  \int_{\Omega_2}u_2\, dx>\int_{\Omega_1}u_1\, dx\Rightarrow T(\Omega_1)>T(\Omega_2).
\end{equation*}

\end{document}